\documentclass{amsart}
\usepackage{amsmath}
\usepackage{amsfonts}
\usepackage{amssymb}
\usepackage{enumerate}
\usepackage[mathscr]{eucal}
\usepackage[all]{xy}
\usepackage{graphicx}
\usepackage{colortbl}
\usepackage{bbm}
\usepackage{verbatim}

\newcommand\bbone{{\mathbbm 1}}
\newcommand{\R}{\mathbb{R}}

\newcommand{\cE}{\mathcal{E}}

\newcommand{\al}{\alpha}

\newcommand{\la}{\lambda}

\newcommand{\oo}{\infty}
\newcommand{\sbs}{\subseteq}

\newcommand{\lss}{\lesssim}
\newcommand{\grs}{\gtrsim}
\newcommand{\wh}{\widehat}
\newcommand{\wt}{\widetilde}

\DeclareMathOperator{\sgn}{sgn}

\newcommand{\Be}{\begin{equation}}
\newcommand{\Ee}{\end{equation}}
\newcommand{\Bea}{\begin{align}}
\newcommand{\Eea}{\end{align}}
\newcommand{\Beas}{\begin{align*}}
\newcommand{\Eeas}{\end{endalign*}}
\newcommand{\Benu}{\begin{enumerate}}
\newcommand{\Eenu}{\end{enumerate}}
\newcommand{\Bi}{\begin{itemize}}
\newcommand{\Ei}{\end{itemize}}

\def\intslash{\rlap{\kern  .32em $\mspace {.5mu}\backslash$ }\int}
\def\qsl{{\rlap{\kern  .32em $\mspace {.5mu}\backslash$ }\int_{Q_x}}}

\def\R{\mathbb R}

\def\floork3{{\lfloor k/3 \rfloor }}
\def\emph#1{{\it #1 }}

\def\bbone{{\mathbbm 1}}

\def\cf{{\it cf}}

\def\sgn{{\text{\rm sign }}}

\def\meas{{\text{\rm meas}}}

\def\lc{\lesssim}

\def\eps{\varepsilon}

\def\la{\lambda}

\def\fN{{\mathfrak {N}}}

\def\bbN{{\mathbb {N}}}

\def\bbR{{\mathbb {R}}}

\def\cE{{\mathcal {E}}}

\newtheorem{theorem}{Theorem}[section]
\newtheorem{prop}[theorem]{Proposition}
\newtheorem{corol}[theorem]{Corollary}

\newtheorem{lemma}[theorem]{Lemma}

\theoremstyle{definition}

\theoremstyle{remark}

\newtheorem*{remarks}{Remarks}

\numberwithin{equation}{section}

\begin{document}

\title[Pointwise convergence of Schr\"odinger means]{On  pointwise convergence of   Schr\"odinger means}


\author{Evangelos Dimou}
\address{E. Dimou, Department of Mathematics, University of Michigan, 530 Church Street
Ann Arbor, MI 48109-1043, USA} 
\curraddr{Department of Mathematics, University of Virginia, P. O. Box 400137, Charlottesville, VA 22904}
\email{ed8bg@virginia.edu}
\thanks{Research supported in part by the National Science Foundation}

\author{Andreas  Seeger}
\address{A. Seeger, Department of Mathematics, University of Wisconsin-Madison, 480 Lincoln Drive, Madison, WI 53706, USA}
\email{seeger@math.wisc.edu}






\begin{abstract} For functions in the Sobolev space $H^s$ and decreasing sequences $t_n\to 0$ we examine  convergence almost everywhere of the generalized Schr\"odinger means on the real line, given by $$S^af(x,t_n)=\exp( it_n (-\partial_{xx})^{a/2})f(x);$$ here $a>0$, $a\neq 1$. For decreasing convex sequences we obtain a simple characterization of convergence a.e. for all functions in $H^s$ when $0<s<\min\{a/4,1/4\}$ and $a\neq 1$.  We  prove sharp quantitative local and global estimates for  the associated maximal functions.
We also obtain sharp results for the case  $a=1$.
\end{abstract}

\maketitle

\section{Introduction}

\medskip
For Schwartz functions  $f$ defined on the real line 
 consider the initial value problem
$$i\partial_tu(x,t)+(-\partial_{xx})^{a/2}u(x,t)=0, \quad u(x,0)=f(x);$$
so that for $a=2$ we recover the Schr\"odinger equation. The solutions are given by
$$S^af(x,t)=\int_{\R}e^{i(x\xi+t|\xi|^a)}\wh f(\xi)\,\tfrac{d\xi}{2\pi},$$ 
and, for fixed time,  the solution operator extends to all $f\in H^s$, where   $H^s$ is the Sobolev space of all distributions  $f$ with $\|f\|_{H_s}:=(\int (1+|\xi|^2)^s|\widehat f(\xi)|^2 d\xi)^{1/2}<\infty$.

One refers to the operators $f\mapsto S^a f(\cdot,t)$ as generalized Schr\"odinger means. For Schwartz functions $f$ it is clear that $\lim_{t\to 0} S^a f(x,t)=f(x)$ and that the convergence is uniform in $x$. 
One is interested in almost everywhere convergence for functions in $H^s$ for suitable $s>0$. Following
the fundamental result by Carleson \cite{Carl}, many authors  have considered this question. It was shown in \cite{Carl}, \cite{Sj} that 
\begin{equation*}
\lim_{t\to 0} S^af(x,t)= f(x) \; \textrm{ a.e.},  \quad f\in H^{1/4},  
\end{equation*} when $a>1$
and this result fails for some $f\in H^s$, if $s<1/4$ (\cite{DK}, \cite{Sj}). 
If $0<a<1$, pointwise convergence for $f\in H^s$ holds when $s>a/4$
 and may fail for $f\in H^s$ when $s<a/4$, see \cite{Wa}. 
We remark that the problem  in higher dimensions is much harder 
and not considered here. For the Schr\"odinger equation in higher dimensions a complete  solution  up to endpoints has been recently found  in \cite{DGL}, \cite{DZ}  
and relies on sophisticated methods from  Fourier restriction theory. We refer to these papers for more references and a historical prospective.

\medskip

In this paper, we consider, in one spatial  dimension,  the question of the solution converging to the initial data when the limit is taken over a decreasing  sequence $\{t_n\}_{n=1}^\infty$, converging to zero. Here we always use the term `decreasing' as synonymous with `nonincreasing'. Given such a sequence we seek to find the precise range of $s$ such that
$\lim_{n\to \infty}\,S^af(x,t_n)=f(x) \ \textrm{a.e.}$
holds for every $f\in H^s$. This is partially 
motivated by the work 
\cite{CLV} on approach regions for pointwise convergence for solutions of the Schr\"odinger equation, and also by the work 
\cite{SWW} on the pointwise convergence of spherical means of $L^p$ functions
(although the mathematical issues and expected outcomes for the latter problem are  different).

 For the class of convex decreasing sequences  and any $s\in (0, \min\{a/4,1/4\})$ we obtain a complete characterization of when pointwise convergence holds for all $f\in H^s$. This characterization involves the  Lorentz space $\ell^{r,\infty}(\mathbb N)$. By definition,  for $0<r<\infty$,  \[
\{t_n\}\in \ell^{r,\infty} \quad \iff \quad   \sup_{b>0} b^{r} \#\{n\in \bbN: |t_n|>b\} <\infty.\] Note that $\ell^{r_1,\infty}(\bbN) \subset \ell^{r_2}(\bbN) \subset \ell^{r_2,\infty}(\bbN)\subset \ell^\infty(\bbN)$ if $r_1< r_2<\infty$ and all inclusions are strict.  A model example is given by $t_n=n^{-\gamma}$ which belongs to $\ell^{r,\infty}$ if and only if $r\ge 1/\gamma$. 
Another  example is  $\{n^{-\gamma}\log n\}$ which belongs to $\ell^{r,\infty}$ if and only if $r> 1/\gamma$. 




\begin{theorem}\label{aeconv} Let $a>0$, $a\neq 1$, and assume $0<s<\min\{a/4, 1/4\}$.
Let $\{t_n\}_{n=1}^\infty$ be a decreasing  sequence with $\lim_{n\to\infty}t_n=0$ and assume that
$t_n-t_{n+1}$ is also decreasing.
 Then the following four statements are equivalent.
 
 \smallskip
 (a) The sequence $\{t_n\}$ belongs to $\ell^{r(s),\infty}(\bbN)$, where $r(s)=\frac{2s}{a-4s}$.
 
  \smallskip
 (b) There is a constant $C_1$ such that for all $f \in H^s$ and for all sets $B$ of diameter at most $1$ we have
 $$\| \sup_{n\in \bbN} | S^af(x,t_n)| \|_{L^2(B)} 
 \le C_1 \|f\|_{H^s}.$$
 \smallskip
 
 (c) There  is  a constant $C_2$ such that for all $f \in H^s$, for all sets $B$ of diameter at most $1$, and for all $\alpha>0$,
 $$ \meas(\{ x\in B: \sup_{n\in \bbN} | S^af(x,t_n)|>\alpha\}) \le C_2\,\alpha^{-2} \|f\|_{H^s}^2.
 $$
  \smallskip
  
 (d) For every $f\in H^s$ we have $$\lim_{n\to \infty}\,S^af(x,t_n)=f(x) \quad \textrm{ a.e.}$$
\end{theorem}

Here and in what follows we write $\meas(A)$ for the Lebesgue measure of  $A\subset\bbR$. The equivalence of (b) and (c) seems nontrivial, and we do not have a direct proof for it, without going through condition (a).
In Theorem \ref{aeconv} the convexity assumption can be dropped for the sufficiency, i.e. statements 
(b), (c), (d) hold 
whenever $t_n$ is decreasing and belongs to $\ell^{\frac{2s}{a-4s},\infty}(\bbN)$,
 see Proposition \ref{endpt} below. 


 
 


Regarding the maximal function inequalities we  also have  a global version:

\begin{theorem}\label{globmax} Let $a>0$, $a\neq 1$, and assume $0<s<a/4$.
Let $\{t_n\}_{n=1}^\infty$ be a decreasing  sequence 
 with $\lim_{n\to\infty}t_n=0$,  and assume that
$t_n-t_{n+1}$ is also decreasing.
 Then the following statements (a), (b), (c)  are equivalent.
 
 (a) The sequence $\{t_n\}$ belongs to $\ell^{\frac{2s}{a-4s},\infty}(\bbN)$.
 
 (b) There is a constant $C_1$ such that for all $f \in H^s$  we have
 $$\| \sup_{n\in \bbN} | S^af(x,t_n)| \|_{L^2(\bbR)} 
 \le C_1 \|f\|_{H^s}.$$

 (c) There  is  a constant $C_2$ such that for all $f \in H^s$ and all $\alpha>0$,
 $$ \meas(\{ x\in \bbR: \sup_{n\in \bbN} | S^af(x,t_n)|>\alpha\}) \le C_2\,\alpha^{-2}  \|f\|_{H^s}^2.
 $$
\end{theorem}




\smallskip

 We contrast the above results with the exceptional case $a=1$ which covers solutions of the wave equation. Now the critical $r(s)=\frac{2s}{a-4s}$ in Theorem \ref{aeconv} has to be replaced with the smaller  $\frac{2s}{1-2s}$, for all $s<1/2$. Notice that $S^1$ corresponds to a family of translation operators, when acting on functions with spectrum in $[0,\infty)$ or $(-\infty,0]$. The analysis is somewhat  similar to   the one for spherical means in \cite{SWW}, see also \cite{stw}.
For $a=1$ we have 

\begin{theorem}\label{aeconva=1} Let $0<s<1/2$ and let $\{t_n\}_{n=1}^\infty$ be a decreasing  sequence with $\lim_{n\to\infty}t_n=0$ such that 
$t_n-t_{n+1}$ is also decreasing.
 Then the following four statements are equivalent.
 
 \smallskip
 (a) The sequence $\{t_n\}$ belongs to $\ell^{\rho(s),\infty}(\bbN)$, where $\rho(s)=\frac{2s}{1-2s}$.
 
  \smallskip
 (b) There is a constant $C_1$ such that for all $f \in H^s$ we have
 $$\| \sup_{n\in \bbN} | S^1 f(x,t_n)| \|_{L^2(\bbR)} 
 \le C_1 \|f\|_{H^s}.$$
 \smallskip
 
 (c) There  is  a constant $C_2$ such that for all $f \in H^s$, for all sets $B$ of diameter at most $1$, and for all $\alpha>0$,
 $$ \meas(\{ x\in B: \sup_{n\in \bbN} | S^1f(x,t_n)|>\alpha\}) \le C_2\,\alpha^{-2} \|f\|_{H^s}^2.
 $$
  \smallskip
  
 (d) For every $f\in H^s$ we have $$\lim_{n\to \infty}\,S^1 f(x,t_n)=f(x) \quad \textrm{ a.e.}$$
\end{theorem}

 The convexity condition is 
  satisfied  for the model case $t_n=n^{-\gamma}$ with $\gamma>0$ and thus Theorems \ref{aeconv}, \ref{aeconva=1} and  the known results for $s=1/4$, when $a>1$,  yield

\begin{corol}\label{corga}
 Let  $0<\gamma<\infty$. 

 (i) If $a>1$, then
$\lim_{n\to \infty}\,S^af(x,n^{-\gamma})=f(x)$ a.e. 
 holds  for every $f\in H^s(\bbR)$ if and only if 
$s\geq\min\{\frac{a}{2\gamma+4},\frac{1}{4}\}$.

(ii) If $0<a<1$, then
$\lim_{n\to \infty}\,S^af(x,n^{-\gamma})=f(x) \; \textrm{a.e.}$ 
  holds for every $f\in H^s(\bbR)$ if and only if 
$s\geq\frac{a}{2\gamma+4}$.

(iii) If $a=1$, then 
$\lim_{n\to \infty}\,S^1f(x,n^{-\gamma})=f(x) \; \textrm{a.e.}$ 
  holds for every $f\in H^s(\bbR)$ if and only if 
$s\geq\frac{1}{2\gamma+2}$.
\end{corol}

This answer for the sequence $\{n^{-\gamma}\}$  reveals a perhaps surprising  phenomenon for the case $a>1$, namely that there is a gain over the general pointwise convergence result when $\gamma>  2(a-1)$, but not when  $0<\gamma\leq 2(a-1)$. When $0<a\le 1$,   we have for all $\gamma\in (0,\infty)$ a gain over the general convergence result. The same remarks apply to the local $H^s\to L^2(B)$ maximal inequality. In contrast we get for the global maximal operator  and 
$a\neq 1$:

\begin{corol}\label{corgaglobal}
 Let  $0<\gamma<\infty$ and $a\in (0,\infty)\setminus\{1\}$. Then the global maximal function inequality
 \[ \|\sup_n |S^af(x,n^{-\gamma})|\|_{L^2(\bbR)}\le C 
\|f\|_{H^s}\]
holds for some $C>0$ and  all  $f\in H^s$ if and only if $s\geq\frac{a}{2\gamma+4}$.
\end{corol}

\begin{remarks}
(i) Our results for the special case $\{n^{-\gamma}\}$, $a\neq 1$ as stated in Corollaries \ref{corga}, \ref{corgaglobal} 
were already incorporated in the 2016 thesis \cite{D} of the first named author. Moreover sufficiency in Theorem \ref{aeconv}, merely for decreasing sequences but under the more restrictive assumption $\{t_n\}\in \ell^r$ for $r< \frac{2s}{a-4s}$, 
follows already from Proposition 1.6 in \cite{D}. 

(ii) The problem of convergence of Schr\"odinger means $S^a(f, t_n)$ for a decreasing sequence $\{t_n\}$  was independently considered  in recent papers by  Sj\"olin \cite{Sj2018} and   
by Sj\"olin and Str\"omberg \cite{SjStr2019}. Their conditions are more restrictive, but apply in all dimensions.
In \cite{Sj2018} it is proved for $a>1$ that the condition
$\{t_n\}\in \ell^{2s/a}$ is sufficient for pointwise convergence. This is  improved in
\cite{SjStr2019} where for $s\le 1/2$, $a>2s$, 
 the condition $\{t_n\}\in\ell^r$ for $r<\frac{2s}{a-2s}$ is shown to be sufficient for pointwise convergence. Proposition \ref{endpt} yields an improvement of these results and 
Theorem \ref{aeconv} gives the  optimal  result for decreasing convex  sequences.


(iii) For $a\neq 1$ there are natural  analogous open questions of necessary and sufficient conditions in higher dimensions, given the recent groundbreaking  results for the full local Schr\"odinger  maximal operator 
in \cite{DGL}, \cite{DZ} which are sharp up to endpoints. 

(iv) For $0<a<1$ there is still the open problem whether  $S^af(x,t)\to f(x)$ a.e. holds for all $f\in H^{a/4}(\bbR)$. Likewise there is the problem of a global bound for the maximal function if $s=a/4$, and $a>1$.
One can show using a variant of the arguments in \cite{Wa}, \cite{Sj2} that
a.e. convergence holds in the  Besov space $B^{a/4}_{2,1}(\bbR)$ which is properly contained in
$B^{a/4}_{2,2}\equiv H^{a/4}$, see Proposition \ref{besovprop}. For the case $a=1$ we have pointwise convergence in $B^{1/2}_{2,1}(\bbR)$, but pointwise convergence fails for some functions in  $H^{1/2}(\bbR)$, see Proposition \ref{besovpropa=1}.
\end{remarks}

\subsection*{\it This paper}
In \S\ref{suffsect} we show for decreasing sequences that the $\ell^{r(s),\infty}$ condition is sufficient for pointwise convergence and the appropriate boundedness properties of the maximal operators. The necessity for decreasing convex sequences (converging to $0$) is proved in \S\ref{negres}.
The case $a=1$ is separately considered in \S\ref{a=1sect}.
In \S\ref{nonaeconv} we include a short appendix  regarding the  relevant  application of Stein-Nikishin theory.

\subsection*{\it Acknowledgement} We would like to thank Per Sj\"olin and Jan-Olov Str\"omberg for their interest and for pointing out an error in a previous version, concerning the term $\cE_1$ in the proof of Proposition \ref{endpt}.

\section{Upper bounds for maximal functions}\label{suffsect}
In the present section we prove maximal function results which 
imply the  positive results of the theorems stated in the introduction. We already know the local estimate
\Be \label{KRSj}\Big\|\sup_{t\in[0,1]}|S^af(\cdot,t)|\Big\|_{L^2(B)}\leq C\|f\|_{H^{1/4}},\Ee
which was established by Kenig and Ruiz  \cite{KR} when $a=2$ and Sj\"olin \cite{Sj} for general $a>1$. 
In view of \eqref{KRSj} it  now suffices  to give the proof of the $L^2(\bbR)$ bound 
in part (b) of Theorem \ref{globmax}, under the assumption of $\{t_n\}\in \ell^{\frac{2s}{a-4s}}$, whenever $s<a/4$.

Throughout this section we assume that $\{t_n\}$ is decreasing but we drop the convexity assumption in the introduction. Without loss of generality (dropping a finite number of terms in the sequence) we can assume that $t_n\in (0,1)$ for all $n\in \bbN$. We first restrict our attention to the frequency localized operator
$$S^a_{\la}f(x,t)=\int_{\R}e^{i(x\xi+t|\xi|^a)}\wh f(\xi)\chi(\xi/\la)\,\tfrac{d\xi}{2\pi},$$
where $\chi\in C^{\oo}$ is a real-valued, smooth  function, supported in $\{1/2\leq|\xi|\leq1\}$. The following   result is a variant of the inequality given in \cite{KR}:

\begin{prop}\label{smint}
If $J\sbs[0,1]$ is an interval and $0<a\neq1$, then
$$\big\|\sup_{t\in J}|S^a_{\la}f(\cdot,t)|\big\|_{L^2(\R)}\leq C(1+|J|^{1/4}\la^{a/4})\|f\|_2.$$
\end{prop}

\medskip

\begin{proof} We use  the Kolmogorov-Seliverstov-Plessner method,   by linearizing the maximal operator: let $x\mapsto t(x)$ be a measurable function, with values in $J$.
It will then suffice to prove
$$\Big(\int_{\R}|S_{\la}^a(x,t(x))|^2\,dx\Big)^{1/2}\leq C(1+|J|^{1/4}\la^{a/4})\|f\|_2,$$ where the constant $C$ is independent of $t(\cdot)$ and $f$.
Notice that
$$S^a_\la f(x,t(x))=\int e^{i(x\xi+t(x)|\xi|^a)}\widehat f(\xi) \chi(\xi/\la)\, \tfrac{d\xi}{2\pi} =\la\,  T^a_\la[\widehat f(\la\cdot)](x)$$
where
$$T_\la^a g(x) =\int e^{i(\la x\xi+\la^a t(x)|\xi|^a)} \chi(\xi)g(\xi)\, \tfrac{d\xi}{2\pi}.$$
Since $\|\widehat f(\la\cdot)\|_2=c\la^{-1/2}\|f\|_2$ we need to show that
$$\|T^a_\la\|_{L^2\to L^2}\lss \la^{(a-2)/4}|J|^{1/4}+\la^{-1/2},$$
 which in turn follows from
\begin{eqnarray}\label{TT*}
\|T^a_\la (T^a_\la)^*\|_{L^2\to L^2}\lss \la^{(a-2)/2}|J|^{1/2}+\la^{-1}.
\end{eqnarray}
The kernel of $T_\la^a (T_\la^a)^*$ is
$$K_\la^a(x,y)= \int e^{i[\la (x-y)\xi+\la^a(t(x)-t(y))|\xi|^a]} \chi^2(\xi)\,\tfrac{d\xi}{2\pi}.$$
and the derivative of the phase $\Phi_\la^a(\xi)=\la (x-y)\xi+\la^a(t(x)-t(y))|\xi|^a$ is equal to
$$(\Phi_\la^a)'(\xi)=\la(x-y)+a\la^a(t(x)-t(y))\,(\sgn \xi)\, |\xi|^{a-1}.$$
Therefore, if $|x-y|\gg\la^{a-1}|t(x)-t(y)|$, we have that $|(\Phi_\la^a)'(\xi)|\grs\la|x-y|$ and integration by parts gives
$$|K_\la^a(x,y)|\lss_N (\la|x-y|)^{-N}.$$
In the case where $|x-y|\lss \la^{a-1}|t(x)-t(y)|$ we use van der Corput's lemma. The second derivative of the phase is
$(\Phi^a_{\la})''(\xi)=c_a\la^a(t(x)-t(y))|\xi|^{a-2}$, hence
$$|K_\la^a(x,y)|\lss \la^{-a/2}|t(x)-t(y)|^{-1/2}\lss(\la|x-y|)^{-1/2}.$$

 Thus $\int_\R|K_\la^a(x,y)|\,dy$ can be estimated by
\begin{align*}
&\int\limits_{\substack{|x-y|\lss\\ \la^{a-1}|t(x)-t(y)|}}\la^{-1/2}|x-y|^{-1/2}dy+
\int\limits_{\substack{|x-y|\gg \\ \la^{a-1}|t(x)-t(y)|}} ( 1+\la|x-y|)^{-N} dy\\
&\quad\le\int_{\substack{|x-y|\lss \la^{a-1}|J|}}\la^{-1/2}|x-y|^{-1/2}dy+\int_{\R} ( 1+\la|x-y|)^{-N} dy\\
&\quad\lss \la^{(a-2)/2}|J|^{1/2}+\la^{-1}\,.
\end{align*}
Therefore $\sup_{x\in\R}\int|K_\la^a(x,y)|\,dy\lss\la^{(a-2)/2}|J|^{1/2}+\la^{-1}$ and by symmetry we get the same bound for
$\sup_{y\in \R}\int|K_\la^a(x,y)|\,dx$.
Hence Schur's test gives the required bound \eqref{TT*}.
\end{proof}


We now use Proposition \ref{smint} to prove a sharp result for the frequency-localized operaors $S^a_\la$.

\begin{lemma}\label{gen} 
Let $0<a\neq 1$, $0<r<\infty$ and let $\{t_n\}$ be a sequence in $[0,1]$ which belongs to $\ell^{r,\infty}$.
 Then for $\la>1$
\[ \Big\|\sup_{n}|S_\la^a f(\cdot,t_n)|\Big\|_{L^2(\R)}\leq C
\la^{\frac{ar}{2+4r}}
\|f\|_{L^2(\bbR)}.\] 
\end{lemma}

\smallskip

\begin{proof} 
We start by writing $$\Big\|\sup_{n}|S_{\la}^af(\cdot,t_n)|\Big\|_{2}\leq\Big\|\sup_{n:\, t_n\leq b}|S_{\la}^af(\cdot,t_n)|\Big\|_{2}+\Big\|\sup_{n:\, t_n>b}|S_{\la}^af(\cdot,t_n)|\Big\|_{2}.$$
By Proposition \ref{smint} we can bound the first term by $b^{1/4}\la^{a/4}\|f\|_2$. On the other hand, by using Plancherel's theorem and our assumption, we get

\begin{align*}\Big\|\sup_{n:\, t_n>b}|S_{\la}^af(\cdot,t_n)|\Big\|_{L^2(\R)}&
\le \Big(\sum_{n:\, t_n>b}\|S_{\la}^af(\cdot,t_n)\|_{2}^2\Big)^{1/2}\\
&\leq\#(\{n: t_n>b\})^{1/2}\|f\|_2\lss b^{-r/2}\|f\|_2.
\end{align*}
We therefore have  $$\Big\|\sup_{n}|S_{\la}^af(\cdot,t_n)|\Big\|_{2}\lss(b^{1/4}\la^{a/4}+b^{-r/2})\|f\|_2$$
and choosing $b$ such that $b^{1/4}\la^{a/4}=b^{-r/2}$, namely $b=\la^{-\frac{a}{1+2r}}$, finishes the proof.
\end{proof}

We wish to apply the Lemma \ref{gen} for $\la=2^k$, $k>1$.  
A more refined argument is needed to combine the dyadic scales.

\begin{prop}\label{endpt} 
Let $0<a\neq 1$, and assume that $\{t_n\}\in \ell^{r,\infty}(\bbN)$ is decreasing.
 Then 
\Be\label{strongtypemaxresult}
\big\|\sup_{n}|S^af(\cdot,t_n)|\big\|_{L^2(\R)}\leq C\|f\|_{H^s}, \quad s= \frac{ar}{2+4r}.\Ee
Moreover $S^af(x,t_n)\to f(x)$ a.e. whenever $f\in H^\sigma$ for $\sigma\ge \min \{\frac 14, \frac{ar}{2+4r}\}$.
\end{prop}

\smallskip

\begin{proof} Define projection operators $P_k$ by
\begin{align*} 
\widehat{P_0 f}(\xi) &=\bbone_{[-1/2,1/2]}(\xi)\widehat f(\xi)
\\
\widehat{P_k f}(\xi) &=(\bbone_{[2^{k-1}, 2^{k}]}+
\bbone_{[- 2^{k}, -2^{k-1}]})\widehat f(\xi), \quad k\ge 1
\end{align*}
Clearly $P_kP_k=P_k$ and $\sum_{k\geq0}P_kf=f$.


Next, for each integer $l\ge 0$ we set 
\[\fN_l=\big\{n\in \bbN: 2^{-(l+1)\frac{a}{1+2r}}
<t_n\le 
2^{-l\frac{a}{1+2r}}\big\}.\]
By assumption $\{t_n\}\in \ell^{r,\infty} $ there is $C>0$ so that
\Be \label{Nlcard}\#(\fN_l)\le C2^{l\frac{ar}{1+2r}}= C 2^{2ls}.\Ee
We can then write
$$\sup_n|S^af(x,t_n)|\leq
\cE_1(x)+\cE_2(x)+\cE_3(x)$$ where
\begin{align*}
\cE_1(x)
&=\sup_l\sup_{n\in\fN_l}\big|\sum_{k< \frac{l}{1+2r}}S^aP_k f(x,t_n)\big|
\\
\cE_2(x)&=\sup_l\sup_{n\in\fN_l}\big|\sum_{\frac{l}{1+2r}\le k<l}S^aP_k f(x,t_n)\big|
\\
\cE_3(x)&=\sup_l\sup_{n\in\fN_l}\big|\sum_{k\ge l}S^aP_k f(x,t_n)\big|
\end{align*}

We first give the estimate for $\|\cE_3\|_2$. We make the change of variable $k=l+m$ and 
and get
\[\cE_3(x)\le \sum_{m\geq0}\Big(\sum_{l\geq0}\sup_{n\in\fN_l}|S^a P_{l+m}f(x,t_n)|^2\Big)^{1/2}.\]
From this,
\begin{align*}
\|\cE_3\|_2 &\le 
\sum_{m\geq0}\Big(\sum_{l\geq0}\Big\|\sup_{n\in\fN_l}|S^aP_{l+m}f(\cdot,t_n)|\Big\|_{2}^2\Big)^{1/2}
\\
&\le \sum_{m\geq0}\Big(\sum_{l\geq0}\sum_{n\in \fN_l} \Big\|S^aP_{l+m}f(\cdot,t_n)\Big\|_{2}^2\Big)^{1/2}
\\
&\le \sum_{m\geq0}\Big(\sum_{l\geq0}\#(\fN_l) \big\|P_{l+m}f\big\|_{2}^2\Big)^{1/2}
\end{align*}
and using \eqref{Nlcard} this is further estimated by
\begin{align*} 
\sum_{m\geq0}\big(\sum_{l\geq0}2^{2sl}\| P_{l+m}f\|_2^2\big)^{1/2}=\sum_{m\geq0}2^{-ms}\big(\sum_{l\geq0}2^{2s(l+m)}\| P_{l+m}f\|_2^2\big)^{1/2}\lss \|f\|_{H^s}.
\end{align*}

In order to deal with the first and second terms
we use that by definition of $\fN_l$ the $t_n$ with $n\in \fN_l$  lie in the  interval $$J_l=[0, 2^{-l\frac{a}{1+2r}}].$$ 

For the term  $\cE_2$ we make the  change of variables $k=l-j$ and estimate
\begin{align*}\cE_2(x)&\le 
\sup_l\sup_{n\in\fN_l}\big|\sum_{0<j\le \frac{2r}{1+2r}l}S^aP_{l-j}f(x,t_n)\big|
\\
&\leq\Big(\sum_{l\geq0}\big(\sum_{0<j\le \frac{2r}{1+2r}l }\sup_{n\in\fN_l}
|S^aP_{l-j} f(x,t_n)|\big)^2\Big)^{1/2}\\
&\leq\sum_{j>0}\big(\sum_{l\ge j\frac {1+2r}{2r}}\sup_{n\in \fN_l}|S^a P_{l-j}f(x,t_n)|^2\big)^{1/2}.
\end{align*}



We can now use  Proposition \ref{smint}, with $J=J_l$ amd $\la=2^k= 2^{l-j}$. 
Note that
$ l\ge j\frac{1+2r}{r} $ implies that $|J_l|^{\frac 14} 
2^{\frac a4 (l-j)}= 2^{-l\frac{a}{4(1+2r)}} 2^{\frac{a}{4}(l-j)}
\ge 1.$
Using  $P_kP_k=P_k$ we then get
\begin{align*}
\|\cE_2\|_2&\le \sum_{j\geq0}\Big(\sum_{ l\ge j\frac{1+2r}{r}}\Big\|\sup_{n\in\fN_l}|S^aP_{l-j}f(\cdot,t_n)|\Big\|_{L^2(\R)}^2\Big)^{1/2}\\
&\lss\sum_{j\geq0}\Big(\sum_{ l\ge j\frac{1+2r}{2r}}\big[
(1+2^{\frac{a}{4}(l-j)}\,2^{-l\frac{a}{4(1+2r)}}) 
\| P_{l-j}f\|_2\big]^2\Big)^{1/2}\\
&\lc\sum_{j\geq0}\Big(\sum_{l\geq j\frac{1+2r}{2r}}\big[2^{(l-j) \frac a4 ( 1-\frac 1{1+2r})} \,2^{-j \frac{a}{4(1+2r)}}
\| P_{l-j}f\|_2\big]^2\Big)^{1/2}\\
&=\sum_{j\geq0}2^{-j \frac{a}{4(1+2r)}}
\Big(\sum_{l\geq j\frac{1+2r}{2r}}\big[2^{s(l-j)}\| P_{l-j}f\|_2\big]^2\Big)^{1/2}\lc \|f\|_{H^s}. 
\end{align*} 

Finally we consider  the term $\cE_1$ and estimate
\begin{align*}
\|\cE_1\|_2 &\le \big\|\sup_l\sup_{n\in\fN_l}
\sum_{k< \frac{l}{1+2r}}|S^aP_k f(x,t_n)|\|_2
\\
&\le 
\sum_{k\ge 0}\big\|\sup_{l>k(1+2r)} \sup_{n\in \fN_l}
|S^aP_kf (\cdot, t_n)|\big\|_2
\end{align*}
The $t_n$ with $n\in \cup_{l>k(1+2r)}\fN_l$ lie in an interval  of length $O(2^{-ka})$, namely in  $J_{l(k)}$ with
$l(k)=\lfloor
 k(1+2r)\rfloor $. 
We  use again  Proposition \ref{smint}, with $J=J_{l(k)}$ and $\la=2^k$, and  observe that
now $2^{ka/4} |J_{l(k)}|^{1/4} \lc 1$. Thus we can bound
for each $k>0$
\[ \big\|\sup_{l>k(1+2r)} \sup_{n\in \fN_l}
|S^aP_kf (\cdot, t_n)|\big\|_2
\lc  (1+2^{ka/4} |J_{l(k)}|^{1/4})
\|P_kf\|_2 \lc \|P_k f\|_2.\]
We sum in $k$ and deduce that
\[ \|\cE_1\|_2\lc\sum_{k\ge 0}  \|P_kf\|_2 
 \le C(s) \|f\|_{H^s}, \quad s>0.\]

We combine the estimates for $\|\cE_i\|$, $i=1,2,3$ to  finish the proof of the maximal inequality \eqref{strongtypemaxresult}. Since 
$\lim_{t\to 0} S^af(x,t)=f(x)$ for all $x\in \bbR$ whenever $f$ is a Schwartz function, and since Schwartz functions are dense in $H^\sigma$ the stated pointwise convergence result follows from \eqref{KRSj}, \eqref{strongtypemaxresult} by a standard argument (see e.g. \cite{Sj} or \cite{SjStr2019}).
\end{proof}

Finally we mention an  endpoint result 
involving the Besov space 
$B^s_{2,1}(\bbR)$ when $s=a/4$. We do not know whether $B^{a/4}_{2,1}$ can be replaced with $H^{a/4}$ in the following proposition.

\begin{prop} \label{besovprop} Let $a>0$, $a\neq 1$. Then, for all $f\in B^{a/4}_{2,1}(\bbR)$,
\[ \big\|\sup_{t\in [0,1]}|S^a f(\cdot,t)|\big\|_{L^2(\R)}\leq C \|f\|_{B^{a/4}_{2,1}}.\]
\end{prop}
\begin{proof} Write $f=\sum_{k\ge 0} S^aP_k P_k f$ as in the proof of Proposition \ref{endpt}. By Proposition \ref{smint} we have
\[ \big\|\sup_{t\in [0,1] }S^af\big\|_2 
 \le\sum_{k\ge 0} \big\|  \sup_{t\in [0,1]}| S^aP_k  P_kf| \big\|_2 \lc 
  \sum_{k\ge 0} 2^{ka/4} \| P_k f\|_2 \]  and using Plancherel's  theorem and the definition of Besov spaces via dyadic frequency decompositions we see that the last expression is  dominated by $C\|f\|_{B^{a/4}_{2,1}}$.
\end{proof}

\section{Necessary conditions}\label{negres}

In order to prove necessity in Theorem \ref{aeconv} we use arguments from Nikishin-Stein theory. We include the standard argument for the proof of the following proposition in Appendix \S\ref{nonaeconv}.
\medskip

\begin{prop}\label{aeconvnec}
 Assume that for every $f\in H^s$, the limit 
 $\lim_{n\to\infty}S^af(x,t_n)$ exists for almost every  $x\in \bbR$.
 Then for any compact set $K\subset \bbR$, there is a constant $C_K$, such that for all $\alpha>0$,
\begin{align*}\label{weakestK}\meas(\{x\in K: \sup_n|S^af(x,t_n)|>\al\})\leq C_K\Big(\frac{\|f\|_{H^s}}{\al}\Big)^2.
\end{align*}
\end{prop}

We also need the following elementary lemma.
\begin{lemma}\label{elem} Let  $\{t_n\}$ be a sequence of positive numbers in $[0,1]$, let $0<r<\infty$  and assume
that $\sup_{b>0} b^r \#(\{n: b<t_n\le 2b\} ) \le A$.
Then $\{t_n\}\in\ell^{r,\infty}$.
\end{lemma}
\begin{proof}
For every $\beta>0$,
\[ \beta^r \#(\{n: t_n>\beta\}) = \beta^r\sum_{k\ge 0} \#(\{n: 2^k\beta<t_n\le 2^{k+1} \beta\})\le 
\sum_{k\ge 0} A 2^{-kr} \lc A. \qedhere\]
\end{proof}

We now turn to the  proof of the necessity of the $\ell^{r,\infty}$-condition in Theorems \ref{aeconv} and Theorem \ref{globmax}. 


\begin{prop} \label{weakestprop}
Assume that $\{t_n\}$ is a decreasing sequence such that $t_n-t_{n+1}$ is also decreasing and  $\lim_{n\to\infty} t_n=0$. For $0<s<a/4$, let
$$r(s)=\frac{2s}{a-4s}.$$
(i) If $s<\min\{a/4,1/4\}$ and if 
\Be\label{weakest01}\meas(\{x\in [0,1] : \sup_n|S^af(x,t_n)|>1/2\})\leq C_\circ \|f\|_{H_s}^2
\Ee  holds for all $f\in H^s$, then $\{t_n\}\in \ell^{r(s),\infty}$.

\smallskip

(ii)  If $s<a/4$ and if the global weak type inequality
\Be\label{weakestglob}\meas(\{x\in \bbR : \sup_n|S^af(x,t_n)|>1/2\})\leq C_\circ \|f\|_{H_s}^2
\Ee holds for all $f\in H^s$, then $\{t_n\}\in \ell^{r(s),\infty}$.
\end{prop}

\begin{proof}
We argue by contradiction and assume that $\{t_n\}\notin \ell^{r(s),\infty}$ 
while \eqref{weakest01} holds if $s<\min\{a/4,1/4\}$ or \eqref{weakestglob} holds in the case $a>1$ and $1/4\le s<a/4$.
By Lemma \ref{elem}, 
this means
\[\sup_{0<b<1/2} b^{r(s)} \#(\{n: b<t_n\le 2b\} ) =\infty.\] Hence
there exists an increasing sequence $\{R_j\}$ with $\lim_{j\to \infty} R_j=\infty$ and a sequence of positive numbers $b_j$ with $\lim_{j\to\infty} b_j= 0$  so that
\Be\label{counterpos}
  \#(\{n: b_j<t_n\le 2b_j\} ) \ge R_j b_j^{-r(s)}.
 \Ee
We take  another sequence 
 $$M_j\le R_j \text{ with } \lim_{j\to\infty} M_j=\infty$$
 such that in the case where $s<1/4$
 \Be \label{conditionMj}
a M_j^{\frac{2(a-1)}{a}} b_j^{\frac{1-4s}{a-4s}}\le 1. 
 \Ee  
 In the case $1/4\leq s<a/4$ we simply take $M_j=R_j$.

 We now show that 
 \Be\label{uppertndiff}
 t_n-t_{n+1} \le 2M_j^{-1} b_j^{\frac{a-2s}{a-4s}},
 \quad \text{ if } ~ t_n\le b_j.
 \Ee
 Indeed since $n\mapsto t_n-t_{n+1} $ is decreasing we get,  for $t_n\le b_j$,
 \begin{align*}
 t_n-t_{n+1}\le \min\{ t_m-t_{m+1}: t_m > b_j\}
 &\le \frac{2b_j} {\#(\{n: b_j< t_n\le 2b_j\})} \\
 &\le \frac{2b_j}{R_j b_j^{-r(s)}}
  \le \frac{2b_j}{M_j b_j^{-r(s)}},
 \end{align*}
 by \eqref{counterpos}, and  \eqref{uppertndiff} follows since $r(s)+1=\frac{a-2s}{a-4s}$.

For our construction of a counterexample we rely on the idea originally proposed by
Dahlberg and Kenig  \cite{DK}. We introduce  a family of Schwartz functions which is used to test \eqref{weakest01}.
Choose a $C^\infty$ function $g$  with compact support in $[-1/2,1/2]$ such that
$g(\xi)\geq 0$ and $\int g(\xi)\,d\xi=1$ and consider  a family of functions 
$f_{\la,\rho}$, with large $\la$ and $\rho\ll \la$,   defined via the Fourier transform by 
$$\widehat f_{\la,\rho}(\eta)= \rho^{-1} g((\eta+\la)/\rho).$$
Thus 
$\widehat f_{\la,\rho}$ is supported in an
interval of length $\rho\ll \la$ contained in $[-2\la,-\la/2]$. 
 The assumption $\rho\ll \la$ clearly implies \Be\label{Hsupperbd} \|f_{\la,\rho}\|_{H^s} \lss \la^s \rho^{-1/2}.\Ee

We now examine the action of $S^a$ on $f_{\la,\rho}$. We have 
\begin{align*}
|S^af_{\la,\rho}(x,t_n)|= \Big|\int e^{i (x\eta+t_n|\eta|^a)}\rho^{-1} g((\eta+\la)/\rho)\,\tfrac{d\eta}{2\pi}\Big|
=
\Big|\int e^{i\Phi_{\la,\rho}(\xi;x,t_n)} g(\xi)\,\tfrac{d\xi}{2\pi}\Big|
\end{align*}
where $$\Phi_{\la,\rho}(\xi;x,t_n)=x(\rho\xi-\la)+t_n(\la-\rho\xi)^a.$$
We shall use, for $x$ in a suitable interval $I_j\subset I$, and for suitable choices of $\la_j,\rho_j$ and $n(x,j)$, 
the estimate 
\begin{align} \label{Saptwlowerbd}|S^af_{\la_j,\rho_j}(x,t_{n(x,j)})|&\geq \int g(\xi)\,d\xi-\int|e^{i\Phi_{\la_j,\rho_j}(\xi;x,t_{n(x,j)})}-1|\, g(\xi)\,\tfrac{d\xi}{2\pi}
\notag
\\&\ge 
1 - \max_{|\xi| \le 1/2} \big|e^{i\Phi_{\la_j,\rho_j}(\xi;x,t_{n(x,j)})}-1\big|
\end{align}
and  we will have to show  that the subtracted  term is small
for our choices of $x$, $n(x,j)$ and $(\la_j,\rho_j) $.

By a standard Taylor expansion, we see that $$(1-\rho\xi/\la)^a=1-a\rho\xi/\la+\tfrac{a(a-1)}{2}(\rho\xi/\la)^2+E_3(\rho\xi/\la)$$
where $E_3(t)=-\frac{1}{2} a(a-1)(a-2) (\int_0^1(1-st)^{a-3}(1-s)^2 ds)\, t^3$. Hence 
\begin{multline}\label{Phiestimate}\Phi_{\la,\rho}(\xi;x,t_n)=\\(x-a\la^{a-1}t_n)\rho\xi+\tfrac{a(a-1)}{2}\rho^2\la^{a-2}t_n\xi^2+\la^a t_n E_3(\rho\xi/\la)+\la^at_n-\la x.\end{multline}
Since terms that are independent of $\xi$ do not affect the absolute value of our integral, we only need to show an upper bound of the first three terms.
We consider $t_n$ with $t_n\le b_j/2$ and let $\eps$ be such that $\eps<10^{-1}(a+2)^{-1}$.
We chose $(\la,\rho)=(\la_j,\rho_j)$ as
\Be\label{lajrhoj}
\la_j = M_j^{2/a}b_j^{-\frac{1}{a-4s}}, 
\qquad \rho_j=\eps b_j^{-1/2} \la_j^{1-a/2} = 
\eps M_j^{\frac{2-a}{a}} b_j^{-\frac{1-2s}{a-4s}}
 \Ee
 and we consider these choices for large $j$ when $b_j\ll 1$ and $M_j\gg 1$. We then get
  \[\rho_j/\la_j=\eps M_j^{-1}b_j^{\frac{2s}{a-4s}}\le \eps;\]
 hence 
  for $|\xi|\le 1/2$
\begin{subequations} 
\begin{align} \label{secondterm}
 &|\tfrac{a(a-1)}{2} \rho_j^2 \la_j^{a-2} t_n \xi^2 |
 \le \tfrac{(a+1)^2}{2} \rho_j^2 \la_j^{a-2} b_j \xi^2 
 \\&\le (a+1)^2\eps^2 M_j^{\frac{4-2a}{a}}b_j^{-\frac{2-4s}{a-4s}} M_j^{\frac{2(a-2)}{a}}b_j^{-\frac{a-2}{a-4s}}b_j= (a+1)^2\eps^2
\notag\end{align}
and similarly 
\Be
\label{thirdterm}
|\la_j^a t_n E_3 (\rho_j\xi/\la)| \le (a+2)^3 \la_j^ab_j\big( \tfrac{\rho_j}{2\la_j}\big)^3
\le (a+2)^3 \eps \la_j^{a-2}\rho_j^2 b_j
\le (a+2)^3\eps^3.
\Ee
Next we consider $x$ in the interval $$I_j:=[0,a\la_j^{a-1} b_j/2].$$ 
Notice that in the case  $s<1/4$,
\[a\la_j^{a-1} b_j/2= \frac a2 M_j^{\frac{2(a-1)}{a}}b_j^{\frac{1-4s}{a-4s}}\le 1/2,\] by \eqref{conditionMj} and hence $I_j\subset [0,1/2]$ in this case. If $a>1$ and $1/4\le s<a/4$, no restriction on $I_j$ is required (as we are trying to disprove the global inequality \eqref{weakestglob}
in this case).
Each $x\in I_j$ is contained in an interval $(a\la_j^{a-1}t_{n+1},a\la_j^{a-1} t_n]$ for a unique $n$, which we label $n(x,j)$. By \eqref{uppertndiff} we have that
$$0\le t_{n(x,j)}-t_{n(x,j)+1}\le 2M_j^{-1}b_j^{\frac{a-2s}{a-4s}}.$$
Hence
\begin{align}
\notag |(x-a\la_j^{a-1}t_{n(x,j)})\rho_j\xi|\le a\la_j^{a-1}\rho_j (t_{n(x,j)}-t_{n(x,j)+1})&
\\\label{firstterm}\le a M_j^{\frac{2(a-1)}{a}} b_j^{-\frac{a-1}{a-4s}} \eps M_j^{\frac{2-a}{a}} b_j^{-\frac{1-2s}{a-4s}} 2M_j^{-1} b_j^{\frac{a-2s}{a-4s}} &=2a\eps.
\end{align}
\end{subequations}
As $\eps\le 10^{-1} (a+2)^{-1}$ we obtain from \eqref{secondterm}, \eqref{thirdterm} and \eqref{firstterm}
\[\max_{|\xi| \le 1/2} |e^{i\Phi_{\la_j,\rho_j}(\xi;x,t_{n(x,j)})}-1| \le 1/2\] and thus 
from \eqref{Saptwlowerbd}
\Be \sup_n|S^af_{\lambda_j,\rho_j}(x,t_n)| \ge |S^a f_{\lambda_j,\rho_j}(x, t_{n(x,j)})|\ge \tfrac 12, \text{ for } x\in I_j=
[0,a\la_j^{a-1} b_j/2],
\Ee
and, as noted before,  $I_j\subset [0,1]$ if $s<1/4$.
The assumption of  \eqref{weakest01} 
(in the case $s<\min\{a/4,1/4\}$) or the assumption of
\eqref{weakestglob},  both yield
\Be \meas(I_j) \le 4C_\circ \|f_{\la_j,\rho_j}\|_{H^s}^2= \widetilde C \la_j^{2s}\rho_j^{-1}.
\Ee 
This leads to
\[aM_j^{\frac{2(a-1)}{a}}b_j^{\frac {1-4s}{a-4s}} \le \widetilde C 
M_j^{\frac {4s}{a}}b_j^{-\frac{2s}{a-4s}} \eps^{-1}M_j^{\frac{a-2}{a}} b_j^{\frac{1-2s}{a-4s}}
\]
and hence to
\[a\eps \widetilde C^{-1} \le M_j^{-\frac{a-4s}{a}}.\]
Since $\lim_{j\to\infty} M_j=\infty$ the  right hand side converges to $0$ as $j\to \infty$ and we obtain a contradiction.
This means that if  $\{t_n\}\notin \ell^{\frac{2s}{a-4s},\infty}$
then
 \eqref{weakest01}  (and therefore \eqref{weakestglob}) cannot hold with  $s<\min\{a/4,1/4\}$  
 and 
\eqref{weakestglob} cannot hold with  $1/4\le s<a/4$.
Thus both parts of the  proposition are  proved.
\end{proof}

We are now able to combine previous results to give a proof of the theorems in the introduction.

\begin{proof}[Proof of Theorem \ref{aeconv}]
The implications (a)$\implies$(b) and (a)$\implies$(d)  follow from Proposition \ref{endpt}.
The implication (b)$\implies$(c) is immediate by Tshebyshev's inequality.
The implication (c)$\implies$(a)   follows from 
part (i)  of Proposition \ref{weakestprop}. Finally, the implication (d)$\implies$(c) follows from Proposition \ref{aeconvnec}.
\end{proof}

\begin{proof}[Proof of Theorem  \ref{globmax}]
The implication (a)$\implies$(b) follows from Proposition \ref{endpt}.
The implication (b)$\implies$(c) is again immediate by Tshebyshev's inequality.
The implication (c)$\implies$(a)   follows from 
part (ii)  of Proposition \ref{weakestprop}.
\end{proof}

\section{The case $a=1$}\label{a=1sect} We now give the sketch of the proof of Theorem \ref{aeconva=1}.
We start with an analog to Lemma \ref{gen}, for the frequency-localized operator $S^1_\la$.

\begin{lemma}\label{gena=1} 
(i) Let $b>\la^{-1}$.
Then \[\big\|\sup_{0\le t\le b} |S_\la^1 f(\cdot,t)|\big\|_2
\lc (\la b)^{1/2} \|f\|_2.\]
(ii) 
Let $0<r<\infty$ and let $\{t_n\}$ be a sequence in $[0,1]$ which belongs to $\ell^{r,\infty}$.
 Then for $\la>1$
\[ \Big\|\sup_{n}|S_\la^1 f(\cdot,t_n)|
\Big\|_{L^2(\R)}\leq C
\la^{\frac{r}{2+2r}}
\|f\|_{L^2(\bbR)}.\] 
\end{lemma}

\smallskip

\begin{proof} We use the elementary inequality
\begin{align} \big\| \sup_{c\le t\le c+\la^{-1}} |S_\la^1 f(x,t)|\big\|_2\lc \|f\|_2
\end{align}
which just follows from $L^2$ estimates for  $S_\la f(\cdot,t)$  and
$\partial_t S_\la f(\cdot,t)$.
Now 
\begin{align*}\Big\|\sup_{0<t\le b}|S_{\la}^1f(\cdot,t)|\Big\|_{L^2(\R)}&
\le \Big(\sum_{\substack{m\ge 0:\\ 0\le m\la^{-1}\le b }}\big\|\sup_{m\la^{-1}\le t\le (m+1)\la^{-1}}| S_{\la}^1f(\cdot,t)|\big\|_{2}^2\Big)^{1/2}
\end{align*}
which is bounded by a constant times 
$(\la b)^{1/2} \|f\|_2$.

To prove part (ii) we write as  in the proof of Lemma \ref{gen},  for $b>\la^{-1}$ to be determined, $$\Big\|\sup_{n}|S_{\la}^1f(\cdot,t_n)|\Big\|_{2}\leq\Big\|\sup_{n:\, t_n\leq b}|S_{\la}^1f(\cdot,t_n)|\Big\|_{2}+\Big\|\sup_{n:\, t_n>b}|S_{\la}^1f(\cdot,t_n)|\Big\|_{2}.$$
For the first term we have
\begin{align*}\Big\|\sup_{n:\, t_n\le b}|S_{\la}^1f(\cdot,t_n)|\Big\|_{L^2(\R)} \lc 
(\la b)^{1/2} \|f\|_2,
\end{align*}
by part (i).
For  the second term we may estimate as in Lemma \ref{gen}
\begin{align*}\Big\|\sup_{n:\, t_n>b}|S_{\la}^1f(\cdot,t_n)|\Big\|_{L^2(\R)}
\lss b^{-r/2}\|f\|_2.
\end{align*}
Choosing $b$ such that $(\la b)^{1/2}= b^{-r}$ yields the claimed result.
\end{proof}

\begin{prop}\label{endpta=1} 
Let $0<r<\infty$ and assume that $\{t_n\}\in \ell^{r,\infty}(\bbN)$ is decreasing.
 Then 
\Be\label{strongtypemaxresulta=1}
\big\|\sup_{n}|S^1f(\cdot,t_n)|\big\|_{L^2(\R)}\leq C\|f\|_{H^s}, \quad s= \frac{r}{2+2r}.\Ee
\end{prop}

\smallskip

\begin{proof} 


We set for  $l\ge 0$ we set 
\[\fN(l)=\big\{n\in \bbN: 2^{-(l+1)\frac{1}{1+r}}
<t_n\le 
2^{-l\frac{1}{1+r}}\big\}.\]
By assumption $\{t_n\}\in \ell^{r,\infty} $ there is $C>0$ so that
\Be \label{Nlcard1}\#(\fN(l))\le C2^{l\frac{r}{1+r}}= C 2^{2ls}.\Ee

Arguing as in the proof of Proposition \ref{endpta=1} we can estimate
$\|\sup_n|S^af(\cdot,t_n)|\|_{2} \le \|\cE_1\|_2+\|\cE_2\|_2$ where 
\begin{align*}
\cE_1(x)&=\sum_{j\geq0}\Big(\sum_{l\geq j}\sup_{n\in\fN(l)}|S^1P_{l-j} f(x,t_n)|^2\Big)^{1/2}
\\
\cE_2(x)&=\sum_{m\geq0}\Big(\sum_{l\geq0}\sup_{n\in\fN(l)}|S^1 P_{l+m}f(x,t_n)|^2\Big)^{1/2}.
\end{align*}
Again as in the proof of Proposition \ref{endpt}
\begin{align*}
\|\cE_2\|_2 &\le \sum_{m\geq0}\Big(\sum_{l\geq0}\#(\fN(l)) \big\|P_{l+m}f\big\|_{2}^2\Big)^{1/2}
\lc \sum_{m\geq0}\big(\sum_{l\geq0}2^{2sl}\| P_{l+m}f\|_2^2\big)^{1/2}\\&=\sum_{m\geq0}2^{-ms}\big(\sum_{l\geq0}2^{2s(l+m)}\| P_{l+m}f\|_2^2\big)^{1/2}\lss \|f\|_{H^s}.
\end{align*}


In order to deal with the first sum, we use that $\fN(l)\subset [0, b_l]$ with $b_l=2^{-l/(1+r)}$. Hence by 
Lemma \ref{gena=1}
\begin{align*}
\|\cE_1\|_2&\le \sum_{j\geq0}\Big(\sum_{l\geq j}\Big\|\sup_{n\in\fN(l)}|S^1P_{l-j}f(\cdot,t_n)|\Big\|_{L^2(\R)}^2\Big)^{1/2}\\
&\lss\sum_{j\geq0}\Big(\sum_{l\geq j}\big[2^{\frac{l-j}{2}}\,2^{-l\frac{1}{2+2r}} 
\| P_{l-j}f\|_2\big]^2\Big)^{1/2}\\
&\lc\sum_{j\geq0}
2^{-j \frac{1}{2+2r}}
\Big(\sum_{l\geq j}\big[2^{(l-j) \frac {r}{2+2r}}\,
\| P_{l-j}f\|_2\big]^2\Big)^{1/2}
\lc \|f\|_{H^s} 
\end{align*} 
with $s=\frac{r}{2r+2}$.
\end{proof}

For completeness we state the 
 case $s=1/2$, $a=1$  analog of  Proposition \ref{besovprop} which is sharp in this case. 

\begin{prop} \label{besovpropa=1} For all $f\in B^{1/2}_{2,1}(\bbR)$,
$$\big\|\sup_{t\in [0,1]}|S^1 f(\cdot,t)|\big\|_{L^2(\R)}\leq C \|f\|_{B^{1/2}_{2,1}}.$$ 
The space $B^{1/2}_{2,1}$ cannot be replaced by $B^{1/2}_{2,\nu}$ for any $\nu>1$.
\end{prop}
\begin{proof} The first part is immediate from Lemma \ref{gena=1}.
For the second part one recalls that there are unbounded functions 
in $B^{1/2}_{2,\nu}$ whose Fourier transform is supported in $(-\infty,0]$, \cf. \cite{BL}.  For such functions $S^1f(x,t)=f(x-t)$ and  thus, for $\nu>1$  one can easily find $f\in B^{1/2}_{2,\nu}$ such that
$\sup_{t\in [0,1]} |S^1f(x,t)|=\infty$ on a set $A\subset [0,1]$  with $\meas(A)>0$.
\end{proof}

\begin{prop} \label{weakestpropa=1}
Assume that $\{t_n\}$ is a decreasing sequence such that $t_n-t_{n+1}$ is also decreasing and  $\lim_{n\to\infty} t_n=0$. For $s<1/2$ let
$$\rho(s)=\frac{2s}{1-2s}.$$
Then the validity of the inequality  
\Be\label{weakest2}\meas(\{x\in [0,1] : \sup_n|S^1f(x,t_n)|>1/2\})\leq C_\circ \|f\|_{H_s}^2.
\Ee  for all $f\in H^s$, implies that $\{t_n\}\in \ell^{\rho(s),\infty}$.\end{prop}

\begin{proof}
Assume that $\{t_n\}\notin \ell^{\rho(s),\infty}$.
Arguing as in the proof of Proposition \ref{weakestprop} we find an increasing sequence $M_j$ with $\lim M_j=\infty$ and  
a sequence of positive numbers $b_j$ with $\lim_{j\to\infty} b_j= 0$  so that
\[\label{counterposa=1}
  \#(\{n: b_j<t_n\le 2b_j\} ) \ge M_j b_j^{-\rho(s)}.
 \]
As in the previous proof we also have
 \Be\label{uppertndiffa=1}
 t_n-t_{n+1} \le 2M_j^{-1} b_j^{\rho(s)+1}, \; \text{ if } ~t_n\le b_j.
 \Ee
 
 Let $g\in C^\infty$ be  nonnegative, supported in $(-1/2,1/2)$, such that $\int g(\xi)\,d\xi=1$.
 Define 
 $f_{\la}$, for large $\la$, by $$\widehat f_\la(\xi) = 
  {10}{\la^{-1}} g(10\la^{-1}(\xi+\la)).$$
 Then $\|f_\la\|_{H^s} \le C_\circ \la^{s-1/2}$. We write 
  \begin{align*}
|S^1f_{\la}(x,t)|= \Big|\int e^{i 10^{-1}\la(x-t)\xi} g(\xi)\tfrac{d\xi}{2\pi}\Big|
\ge 1 -
\Big|\int (e^{i 10^{-1} \la(x-t)\xi}-1) g(\xi)\,\tfrac{d\xi}{2\pi}\Big|
\end{align*}
and   see that 
\Be \label{ptwlowerbda=1}
|S^1f_\la(x,t)|\ge 1/2, \quad \text{  if $|x-t|\le \la^{-1}$.}
\Ee
We now set $\la_j=M_j b_j^{-\frac{1}{1-2s}}$. Note that $\rho(s)+1=(1-2s)^{-1}$, and thus we have $t_n-t_{n+1}\le M_j \la_j^{-1}$ for $t_n\le b_j$, by \eqref{uppertndiffa=1}. Hence by  \eqref{ptwlowerbda=1}
we see that 
$$\sup_n |S^1f_{\la_j} (x) |\ge 1/2, \quad \text{ for $0<x<b_j/2$.}$$
Therefore the asserted weak type inequality implies 
$$b_j/2 \le 4 \|f_{\la_j}\|_{H^s}^2 \le 4 C_\circ \la_j^{2s-1}= 4C_\circ M_j^{2s-1} b_j $$
and thus $8C_\circ M_j^{2s-1}$ is bounded below as $j\to \infty$. This  yields a contradiction as we have $\lim_{j\to 0} M_j^{2s-1}=0$ for $s<1/2$.
\end{proof}

\begin{proof}[Proof of  Theorem \ref{aeconva=1}]
The implication (a)$\implies$(b) follows from Proposition \ref{endpta=1}. The implication (b)$\implies$(c) follows from Tshebyshev's inequality. The implication (c)$\implies$(a) follows from Proposition \ref{weakestpropa=1}. The implication (c)$\implies$(d)  follows by a standard argument using the weak type inequality and the density of Schwartz functions. The implication (d)$\implies$(c) follows from  Proposition \ref{aeconvnec}.
\end{proof}

\section{Appendix: Proof of Proposition \ref{aeconvnec}}\label{nonaeconv}

 We need to use a  theorem by Nikishin, whose proof can be found, for example, in \cite{RdF-GC} (Chapter VI, Corollary 2.7), see also \cite{steinsequences}.
Nikishin's theorem asserts that if $M:L^2(Y,\mu)\to L^0(\R^d,|\cdot|)$ is a continuous sublinear operator (with $(Y,\mu)$ an arbitrary measure space), then there exists a measurable function $w$ with  $w(x)>0$ a.e. such that
$$\int_{\{x:\,|Mf(x)|>\alpha\}}w(x)\,dx\le\alpha^{-2} \|f\|_{L^2(\mu)}^2.$$


\medskip

To prove Proposition \ref{aeconvnec}, let $M^af(x)=\sup_n|S^af(x,t_n)|$ and consider 
$T_n^ag(x)=(2\pi)^{-1}\int e^{i(x\xi+t_n|\xi|^a)}g(\xi)\,d\xi$, so that 
$T_n^a \widehat f(x)=S^af(x,t_n)$. Then $T_n^a$ acts on functions in the weighted $L^2$ space $L^2(\mu_s)$, where $d\mu_s(\xi)=(1+|\xi|^2)^{s}d\xi$. Define  the corresponding maximal operator, $\wt M^ag=\sup_n|T_n^ag|$. 

Now assuming that $\lim_n S^af(x,t_n)$ exists a.e. for every $f\in H^s$, we see that $\wt M^ag(x)<\infty$ a.e. for every $g\in L^2(\mu_s)$. Then by Proposition 1.4, p. 529 in \cite{RdF-GC}, this implies that the sublinear operator $\wt M^a: L^2(\mu_s)\to L^0(|\cdot|)$ is continuous. By the abovementioned  Nikishin's theorem,
$$\int_{\{x:\,|\wt M^ag(x)|>\alpha\}}w(x)\,dx\le\alpha^{-2} \|g\|_{L^2(\mu_s)}^2$$
for some weight $w$ with $w(x)>0$ a.e. As we can replace $w$ with $\min\{w,1\}$ we may further assume that $w$ is bounded.

Next, for $f\in H^s$,  $\wt M^a\widehat f=M^a f$ and $\|\widehat f\|_{L^2(\mu_s)}=\|f\|_{H^s}$, so
$$\int_{\{x:\,|M^af(x)|>\alpha\}}w(x)\,dx\le\alpha^{-2} \|f\|_{H^s}^2.$$ After the change of variables $x\to x+y$ and using the translation invariance of $M^a$, we can replace the integrand $w(x)$ by $w(x-y)$ for any $y$. Arguing as in Chapter VI of \cite{RdF-GC} 
multiply both sides of the resulting inequality by $h(y)$, where $h$ is a strictly positive continuous function with $\int h=1$, and then integrate in $y$, to arrive at
$$\int_{\{x:\,|M^af(x)|>\alpha\}}h*w(x)\,dx\le\alpha^{-2} \|f\|_{H^s}^2.$$
Since 
$h*w$ is  continuous  it attains a minimum over any compact set. We therefore conclude that $$\meas\big(\{x\in K: |M^af(x)|>\alpha\}\big)\le C_K \alpha^{-2} \|f\|_{H^s}^2$$ must hold true for every compact set $K$, as desired. \qed


{}

\end{document}